\documentclass{amsart}
\usepackage{times}
\usepackage{amsfonts}
\usepackage{amsmath}
\usepackage{amscd}
\usepackage{amssymb}
\usepackage{amsxtra}
\usepackage{latexsym}
\usepackage{amsthm}
\usepackage{wasysym}
\usepackage[bookmarks,colorlinks,pagebackref]{hyperref} 
\input{xy}
\xyoption{all}

\usepackage{graphicx}

\theoremstyle{plain}

\newtheorem*{conjectuur*}{Conjecture}

\swapnumbers
\newtheorem{theorem}[subsection]{Theorem}
\newcommand\Thm[1]{Theorem~\ref{#1}}
\newtheorem{corollary}[subsection]{Corollary}
\newcommand\Cor[1]{Corollary~\ref{#1}}
\newtheorem{lemma}[subsection]{Lemma}
\newcommand\Lem[1]{Lemma~\ref{#1}}
\newtheorem{proposition}[subsection]{Proposition}
\newcommand\Prop[1]{Proposition~\ref{#1}}

\theoremstyle{definition}

\theoremstyle{remark}

\newtheorem{remark}[subsection]{Remark}

\swapnumbers

\newcommand{\emptyprop}{q}

\newcommand \binomial[2]{{\bigl( \begin{matrix} #1\cr#2\cr\end{matrix} \bigr)}}
\newcommand\ch{characteristic}
\newcommand \CM{Coh\-en-Mac\-au\-lay}

\renewcommand \hom [3]{\operatorname{Hom}_{#1}(#2,#3)} 
\newcommand \homo{homomorphism}
\newcommand \id{\mathfrak a}
\renewcommand\iff{if and only if}

\newcommand \inv[1]{{#1^{-1}}}
\newcommand \inverse[2]{{#1^{-1}(#2)}}
\newcommand \iso{\cong}

\newcommand \map[1]{{\newcommand{\tmpprop}{#1q}  \if\tmpprop\emptyprop \to\else \xrightarrow{{\phantom{i}{#1}\phantom{i}}}\fi}} 
\newcommand \maxim{\mathfrak m}
\newcommand \nat{\mathbb N}

\newcommand \op\operatorname
\newcommand \pol[2]{#1[#2]}
\newcommand \pow[2]{#1[[#2]]}

\newcommand \pr{\mathfrak p}

\newcommand \range [2]{#1,\dots,#2}

\newcommand \rij[2]{(#1_1,\dots,#1_{#2})}
\let\sub\subseteq

\newcommand \tensor{\otimes}

\newcommand \zet{\mathbb Z}

\newcommand \exactseq [5]{0\to{#1}\:\map{#2}\:{#3}\:\map{#4}\:{#5}\to0}

\newcommand\lc[2]{\hlc {#1}{#2}\maxim}
\newcommand\rhom[2]{\underline{\op{RHom}}(#1,#2)}
\newcommand\rlc{\underline{\op{R}\!\Gamma}_\maxim}
\newcommand\dercomp{\underline\omega_R}
\newcommand\derdual{\underline{\op D}}
\newcommand\KSker[1]{\mathfrak {K}(#1)}

\newcommand \powers[2]{ #1^{(#2)}}
\newcommand\der[2]{\op{Der}_{#1}(#2)}
\newcommand\ndomod[2]{\op{End}_{#1}(#2)}

\newcommand\can[2]{\canring{#1}{#2}{}}
\newcommand\canring[3]{\mathbf K_{#2}^{#3}(#1)}
\newcommand\matlis[1]{#1^\vee}
\newcommand\hlc[3]{\op{H}_{#3}^{#2}(#1)}
\newcommand\frob[1]{\mathbf{F}_{#1}}
\newcommand\tuple[1]{\mathbf{#1}}

\title[A differential-algebraic criterion for obtaining a small MCM]{A differential-algebraic criterion for obtaining a small maximal Cohen-Macaulay module}
\author{Hans Schoutens}
\date\today
\address{Department of Mathematics\\
NYC College of Technology and
the CUNY Graduate Center\\
New York, NY, USA}
\begin{document}
\begin{abstract} 
We show how for a three-dimensional complete local ring  in positive \ch, the existence of an F-invariant, differentiable derivation   implies  Hochster's small MCM conjecture. As an application we show that any three-dimensional pseudo-graded ring in positive \ch\ satisfies Hochster's small MCM conjecture.
\end{abstract}
\maketitle


\section{Introduction}
Hochster observed that almost all of the homological conjectures over a Noetherian local ring $(R,\maxim)$ would follow readily from the existence of a  \emph{maximal \CM} module  (MCM, for short), that is to say,    a module whose depth is equal to the dimension of $R$ (for an overview, see \cite{HoCurr}).  If $R$ is complete,\footnote{Since all homological conjectures admit faithfully flat descent, there is no loss of generality in proving the existence of MCM's after taking a scalar extension (in the sense of \cite[\S3]{SchClassSing}), and so we may assume that $R$ is furthermore complete and has algebraically closed residue field. Moreover, one may always kill a  prime ideal of maximal dimension and assume in addition that $R$ is a domain.\label{f:comp}} being an MCM is equivalent with any system of parameters of $R$ becoming a regular sequence on the module. Together with Huneke, he then proved their existence in equal \ch\ (\cite{HHbigCM}, with a simplified proof in \ch\ zero using ultraproducts by the author in \cite{SchBCM}). Recent work of Andr\'e has now also settled the mixed \ch\ case (\cite{AndDS,AndMix}). Around the same time, he also asked whether in the complete case, we can even   get a \emph{small} (=finitely generated) MCM. However, for dimension three and higher, the latter remains largely an open question.\footnote{Even Hochster has now expressed doubt about the truth of this conjecture.} 

In \cite{SchMCMFsplit}, I gave a new condition (involving local cohomology) for the existence of a small MCM, and deduced the conjecture for three-dimensional F-split complete local rings. In the present paper, I will extend this to some other  three-dimensional complete local rings $(R,\maxim)$. Henceforth, we will in addition assume that $R$ is  a domain  with algebraically closed residue field $k$ (see footnote~\ref{f:comp}). I will describe the numerical invariant $h$ from the cited paper and review the argument how the existence of a small MCM follows from the vanishing of $h$ on some unmixed module (\Prop{P:candep2}). For the remainder of this introduction, we now assume that $R$   has moreover positive \ch, so that we can use Frobenius transforms. Since $h$ is invariant under Frobenius transform and is additive on direct sums, the problem reduces to finding `enough' \emph{F-decomposable} modules, that is to say, modules whose Frobenius transform is decomposable (see \Prop{P:FaddFind} below for a precise statement). For instance, the main result of \cite{SchMCMFsplit} is an instance of this principle, as F-purity means that $R$ is a direct summand of  $\frob*R$.   In \S\ref{s:diff}, I then introduce some techniques from differential algebra and deduce the main theorem: if $R$ admits a Hasse-Schmidt derivation $(1,H_1,H_2,\dots)$ with $H_1^p=H_1$, then it admits a small MCM. 

The last section is then devoted to a special type of rings for which these Hasse-Schmidt derivations always exist, and therefore satisfy Hochster's small MCM conjecture in dimension three: the class of pseudo-graded rings. These include two of the previously known cases of Hochster's conjecture:  
\begin{itemize}
\item the completion  of a three-dimensional   graded ring;\footnote{Hochster attributes this case independently to Hartshorne and Peskine-Szpiro, see \cite{HoCurr}.} 
\item an analytic toric singularity, that is to say,  the completion  of   the coordinate ring of a  point on a toric variety.\footnote{The normalization of an analytic toric singularity is   a small MCM by \cite{HoToric}; it is pseudo-graded since its ideal of definition can be generated by binomials by  \cite{EisStu}.} 
\end{itemize}
However, we can now construct new examples from these: e.g., take a hypersurface in a four-dimensional analytic toric singularity with defining equation given by a quadrinomial (=polynomial with four non-zero terms), then each irreducible component is pseudo-graded, whence admits a small MCM (\Cor{C:tor4quad}).  I have also included an appendix, in which I make the connection between differential operators and F-decomposability more explicit, which hopefully gives more credence to this particular approach to solve Hochster's small MCM conjecture in dimension three.

\section{A cohomological criterion and F-decomposability}

Throughout, fix a   complete local domain $(R,\maxim)$ of dimension $d\geq 2$,  with algebraically closed residue field $k$ (see footnote~\ref{f:comp}), and let $M$ be a finitely generated $R$-module. Let $E$ be the injective hull of $k$ and denote the Matlis dual of a module $Q$ by $\matlis Q:=\hom RQE$.
\subsection*{Local cohomology}
We use the following facts
\begin{enumerate}
\item\label{i:supplc}   $\lc Mi $ is non-zero for $i$ equal to   $\op{depth}(M)$ and $\op{dim}(M)$, and some values in between (\cite[Theorem 3.5.7]{BH});
\item each $\lc M  i$ is Artinian  (\cite[Lemma 3.5.4]{BH}), whence $\matlis{\lc M  i}$ is finitely generated, and has dimension   at most $i$ (\cite[Proposition 2.5]{SchMCMFsplit}). 
\end{enumerate}
Put $\can M{}:=\can MR:=\matlis{\lc M d}$. If $S\sub R$ is a Noether normalization (i.e., a finite extension with $S$ regular), then $\can R{}=\hom SRS$, showing that in general   $\can R{}$ is unmixed. In fact, if $R$ is \CM, then $\can R{}$ is its canonical module.
From $\lc M d\iso M\tensor \lc R d$, we get 
\begin{equation}\label{eq:toppcan}
\can M{}=\hom RM{\can R{}}
\end{equation} 
 whenever $d=\dim M$, so that the unmixedness of  $\can R{}$ implies that of  $\can M{}$. The following invariant will play an important role in the sequel:
\begin{equation}\label{eq:h}
h(M):=\ell(\lc {\matlis{\lc M{d-1}}}0)<\infty.
\end{equation}

\begin{proposition}\label{P:candep2}
Let $M$ be a $d$-dimensional module, then $\can MR$ 
 has always depth at least two, and even depth at least three whenever $h(M)=0$.
\end{proposition} 
\begin{proof}
Let $D:=\lc M0\sub M$. As $\lc Dd=\lc D{d-1}=0$ by \eqref{i:supplc},   the long exact sequence of local cohomology yields $\can M{}=\can{M/D}{}$, and so upon replacing $M$ by $M/D$, we may assume that $M$ has positive depth. Choose an $M$-regular element $a\in\maxim$, and let $\bar R:=R/aR$ and $\bar M:=M/aM$. The short exact sequence $\exactseq MaM{}{\bar M}$ yields a long exact sequence
$$
\dots \to \lc{\bar M}{d-1}\to \lc Md\map a\lc Md\to \lc {\bar M}d=0 
$$ 
and hence taking Matlis duals, we get an exact seqeunce
\begin{equation}\label{eq:canMx}
0\to \can M{}\map a\can M{}\to \matlis{\lc{\bar M}{d-1}}\iso \can {\bar M}{\bar R}
\end{equation} 
where the last isomorphism follows since $\bar M$ has dimension $d-1$ over $\bar R$. Since $\can{\bar M}{\bar R}$ is unmixed and  contains $\can M{}/a\can M{}$ as a submodule,   the latter is also unmixed. As $a$ is $\can M{}$-regular, the first assertion follows. 

If $h(M)=0$, then $K':=\matlis{\lc M{d-1}}$ has also positive depth, and so we may  choose $a$ to be in addition $K'$-regular (by prime avoidance).  The exact sequence \eqref{eq:canMx} extends to
\begin{equation}\label{eq:canMxlong}
0\to \can M{}\map a\can M{}\to \can {\bar M}{\bar R}\to K'\map a K'
\end{equation} 
showing that $\can M{}/a\can M{}\iso \can{\bar M}{\bar R}$. By the first assertion, the latter has depth at least two (as an $\bar R$-module, whence as an $R$-module), and so we are done. 
\end{proof} 

From now on, we will also assume that $R$ has   \ch\ $p>0$. Let $\frob p$ denote the Frobenius map and write $\frob *M$ for the pull-back of $M$ along $\frob p$. That is to say, think of $\frob *M$ as having elements $*x$, for $x\in M$, with scalar multiplication by an element $a\in R$ given by $a{*}x:=*a^px$ (and addition as in $M$). Our assumptions on $R$  imply that $\frob *M$ is again a finitely generated $R$-module, called the \emph{Frobenius transform} of $M$. We have
\begin{enumerate}
\addtocounter{enumi}2
\item\label{i:frobfinlen} if $D$ has finite length, then $\ell (D)=\ell(\frob*D)$;
\item\label{i:frobcommmat} Frobenius transforms commute with local cohomology and Matlis duality.
\end{enumerate}
For \eqref{i:frobfinlen}, note that $k$ being algebraically closed implies $\frob*k\iso k$, and the rest now follows by induction on $\ell (D)$ and exactness of $\frob*$. The first statement of \eqref{i:frobcommmat} follows from the \v{C}ech perspective of local cohomology and the second is proven in \cite[Theorem 4.6]{SchMCMFsplit}).\footnote{
The following shorter argument using derived categories was proposed by an anonymous reviewer. 
Let $\dercomp$ be the normalized dualizing complex and write $\derdual(-) := \rhom -\dercomp$ for the Grothendieck dual, so that in particular $\rlc(\dercomp)\iso E$. Since $\frob*=\frob!$ commutes with $\rlc$, it also commutes with $\derdual(-)\iso \rhom{\rlc(-)}E$ by Grothendieck duality, and finally also with Matlis duality, since 
$$
\rlc(\derdual(-)) \iso \rhom -{\rlc(\dercomp)}\iso  \rhom -E = \matlis -.
$$
}
From this,  we get 
\begin{equation}\label{eq:hft}
h(M)=h(\frob*M).
\end{equation} 
This already yields the main result from \cite{SchMCMFsplit}: if $R$ is a three-dimensional complete F-pure ring, then it admits a small MCM. Indeed, F-purity implies that $\frob*R\iso R\oplus Q$ for some (finitely generated) $Q$, and using \eqref{eq:hft} we get $h(R)=h(\frob*R)=h(R)+h(Q)$, whence $h(Q)=0$, so that we can apply \Prop{P:candep2}.

\subsection*{F-decomposability}
We say that $M$ is \emph{F-decomposable}, if some $\frob*^nM$ is decomposable. Some examples are non-simple modules of finite length, and F-pure rings (see the previous paragraph).\footnote{In fact, I suspect that $k$ is the only F-indecomposable module---note that if this were true, then \Prop{P:FaddFind} below would  apply to the class of all unmixed modules, proving Hochster's conjecture in dimension three.} Any direct summand of some $\frob*^nM$ will be a called an \emph{F-component}. 

\begin{proposition}\label{P:FaddFind}
If there  exists a nonempty   class $\mathcal H$   of F-decomposable unmixed modules which is closed under F-components (i.e., closed under Frobenius transforms and direct summands), then there is some   $M\in\mathcal H$ with $h(M)=0$.
\end{proposition}
\begin{proof}
Choose $M\in\mathcal H$ with $h(M)$ minimal. By assumption, there is some $n$ such that $\frob*^nM\iso P\oplus Q$, with $P,Q\in\mathcal H$. Since $h$ is additive on direct sums, \eqref{eq:hft} yields
$$
h(P)+h(Q)=h(\frob*^nM)=h(M),
$$ 
so that by minimality, we must have $h(M)=0$.  \end{proof}  

\section{Hasse-Schmidt derivations and F-decomposability}\label{s:diff}
Recall that a derivation $D$ on $R$ is a $k$-linear map satisfying the Leibniz rule $D(ab)=aD(b)+D(a)b$, for $a,b\in R$. Repeating this rule, we get
\begin{equation}\label{eq:genLeib}
D^p(ab)=\sum_{i=0}^p \binomial pi D^i(a)D^{p-i}(b)=aD^p(b)+D^p(a)b
\end{equation} 
proving that $D^p$ is again a derivation; if $D^p=D$ then we say that $D$ is  \emph{F-invariant}.

Recall that a $k$-linear endomorphism $f$ of $M$ is said to have \emph{order at most $n$}, if, by recursion,  $[f,a]$ has order at most $n-1$, for all $a\in R$, where we declare the elements of $R$, identified with the left multiplication maps on $M$, as having order zero. Endomorphisms of finite order are then called   \emph{differential operators}. Let $f$ be a linear differential operator on $M$. This means that for each $a\in R$, there exists $d_a\in R$, such that $d_ax=[f,a](x)=(fa-af)(x)=f(ax)-af(x)$. Since the map $a\mapsto d_a$ corresponds to the restriction of $[f,-]$ on $R$, it is in fact a derivation $\delta_M(f)\in \der kR$. In other words, $\delta_M(f)=D$ means that $f(ax)=af(x)+D(a)x$, for all $a\in R$ and $x\in M$ (one says that $f$ is a \emph{$D$-skew} derivation). Let $\KSker M\sub \der {}R$ be the image of $\delta_M$ (it is called the \emph{Kodaira-Spencer kernel} of $M$ as it can be realized as a kernel between Hochschild homology \cite[\S9]{WeiHom}).  We have  
\begin{equation}\label{eq:KSsum}
\KSker{P\oplus Q}=\KSker P\cap  \KSker Q.
\end{equation} 
Indeed, let $\pi$ denote  the projection $M:=P\oplus Q\to P$. Given $D\in \KSker{M}$, choose a linear differential operator $f$ on $M$ with $\delta_M(f)=D$, and let $p\colon P\to P\colon x\mapsto \pi(f(x))$. One easily verifies  that $[p,a]=D(a)$, for $a\in R$, showing that $\delta_P(p)=D\in\KSker P$. Conversely, if $D\in \KSker P\cap  \KSker Q$, then we can find linear differential operators $f$ and $g$ on $P$ and $Q$ respectively with $\delta_P(f)=D=\delta_Q(g)$. It is now easy to see that $f\oplus g$ is a linear differential operator on $M$ with $\delta_M(f\oplus g)=D$.

\begin{proposition}\label{P:FinvDind}
If there is an F-invariant  $D\in \KSker M$, then $\frob*M$ is decomposable. 
\end{proposition} 
 \begin{proof}
Suppose not, so that $\mathcal E:=\ndomod R{\frob*M}$ is hereditary strongly local with residue field $k$ by \Thm{T:indloc}. Choose a linear differential operator $f$ on $M$ with $\delta_M(f)=D$. As in \eqref{eq:genLeib}, we have
$$
f^p(ax)=\sum_{i=0}^p \binomial pi D^i(a)f^{p-i}(x)=af^p(x)+D^p(a)b\quad\text{for all $a\in R$ and $x\in M$,}
 $$
 so that $f^p$ has   order one and $\delta_M(f^p)=D^p=D$. Put $\phi:=f^p-f$. Since $\delta_M(f^p-f)=0$, we get $[\phi,a]=0$, for all $a\in R$, which means that $\phi$ is $R$-linear. Let $S$ be the subalgebra of $\ndomod RM\sub\mathcal E$ generated by $\phi$, so that $S$ is commutative, local and complete with residue field $k$ ( by \Thm{T:indloc}). The Artin-Schreier polynomial $P(T):=T^p-T-\phi\in\pol ST$ has a solution in $k$, since $P(f)=0$. Since it is an etale equation, it therefore has already a solution $\tau\in S$ (by Hensel's lemma). Hence $(f-\tau)^p=f^p-\tau^p=f+\phi-(\tau+\phi)=(f-\tau)$, so that $(f-\tau)^{p-1}$ is a (non-trivial) idempotent in $\mathcal E$, contradicting \eqref{i:idem}.  
\end{proof} 

\subsection*{Hasse-Schmidt derivations}
Recall that a \emph{Hasse-Schmidt derivation}  on $R$ is a sequence $\mathbf H$ of $k$-linear endomorphisms $H_l$, for $l\in\nat$, such that   $H_0=1$ and 
\begin{equation}\label{eq:mixLeib}
H_l(ab)=\sum_{i=0}^l H_i(a)H_{l-i}(b)\quad\text{for all $a,b\in R$ and $l\geq 0$}.
\end{equation} 
Putting $\Phi_{\mathbf H}:=\sum H_lt^l$, viewed as a $\pow kt$-linear endomorphism on $\pow Rt$, then $\Phi_{\mathbf H}$ is multiplicative and reduces to the identity modulo $t$, whence  is an automorphism of $\pow Rt$, and conversely any such automorphism induces a unique Hasse-Schmidt derivation (see, for instance, \cite[\S27]{Mats}). In particular, $H_1$ is a derivation on $R$, and more generally, $H_l$ is a differential operator of order at most $l$. If an ideal $I\sub R$ is $\mathbf H$-invariant (i.e., $H_l(I)\sub I$, for all $l$), then $\mathbf H$ induces a Hasse-Schmidt derivation on $R/I$. 
We call an arbitrary derivation $D$ \emph{integrable}, if there exists some Hasse-Schmidt derivation $\mathbf H$ with $H_1=D$.  

%

\begin{corollary}\label{C:HSass}
Given a prime ideal $\pr\sub R$, then a Hasse-Schmidt derivation   $\mathbf H$  on $R$  induces one on $R/\pr$ in the  following two cases:
\begin{enumerate}
\item\label{i:HSass} $\pr$ is an associated prime of $R$;
\item\label{i:HSsing} $\pr$ is a minimal prime of the singular (respectively, non-\CM, non-Gorenstein, non-normal) locus of $R$.
\end{enumerate}
\end{corollary} 
\begin{proof}
We need to show that in either case $\pr$ is $\mathbf H$-invariant, which amounts to showing that  the $\pow kt$-algebra automorphism  
  $\Phi_{\mathbf H}$  of $\pow Rt$ defined by $\mathbf H=(H_i)_i$ preserves $\pr\pow Rt$. In case \eqref{i:HSass}, since $\pr\pow Rt$ is then an associated prime of $\pow Rt$, so must its image $\Phi_{\mathbf H}(\pr\pow Rt)$ be. But any associated prime of $\pow Rt$ is extended from $R$, that is to say, of the form $\mathfrak q\pow Rt$, for some associated prime $\mathfrak q$ of $R$. Since $\Phi_{\mathbf H}(a)=a+tf$, for some $f\in \pow Rt$, we see that  $a\in\pr$ implies $a\in\mathfrak q$, that is to say, $\pr\sub \mathfrak q$. Reasoning instead with the inverse of $\Phi_{\mathbf H}$, we get the other inclusion, showing that $\Phi_{\mathbf H}(\pr\pow Rt)= \pr\pow Rt$.

In case \eqref{i:HSsing}, let $\id$ be the radical ideal defining the singular (respectively, non-\CM, non-Gorenstein, non-normal) locus  of $R$, and let $J:=\Phi_{\mathbf H}(\id\pow Rt)$. Let $\mathfrak Q\sub\pow Rt$ be a prime ideal and set $\mathfrak P:=\inverse{\Phi_{\mathbf H}}{\mathfrak Q}$. If $\mathfrak Q$ does not contain $J$, then $\mathfrak P$ does not contain $\id\pow Rt$. As the latter ideal defines the singular (respectively, non-\CM, non-Gorenstein, non-normal) locus of $\pow Rt$, we see  that  $\pow Rt_{\mathfrak P}$
is regular (respectively, \CM, Gorenstein, normal), whence so is $
  \pow Rt_{\mathfrak Q}$ under the isomorphism $\Phi_{\mathbf H}$. This proves that $\id\pow Rt\sub J$. Conversely, if $\pow Rt_{\mathfrak Q}$ is regular (respectively, \CM, Gorenstein, normal), then so is $\pow Rt_{\mathfrak P}$ and hence $\mathfrak P$ does not contain $\id\pow Rt$, whence neither does $\mathfrak Q$ contain $J$, showing that $J\sub \id\pow Rt$, so that  $\Phi_{\mathbf H}$ preserves the ideal $\id\pow Rt$.
    As $\pr$ is a minimal prime of  $\id$, the result follows from \eqref{i:HSass}. 
\end{proof}

\begin{proposition}\label{P:HSFrob}
If $\mathbf H$ is a Hasse-Schmidt derivation, then $H_1\in \KSker{\frob*^nR}$, for all $n$. 
\end{proposition}
\begin{proof}
Replacing $p$ by some power, we easily reduce to the case that $n=1$, and so we are done once we show that $f:=\frob*(H_p)$ is a linear differential operator on $\frob*R$ with $\delta(f)=H_1$. For $a\in R$ and $x:=*b\in \frob*R$, we have
$
f(ax)=f(*a^pb)=*H_p(a^pb).
$
To calculate $H_p(a^pb)$, we need, with $\Phi$ the automorphism given by $\mathbf H$,  the coefficient of $t^p$ in  
$$
\Phi(a^pb)=\Phi(a)^p\Phi(b)=(a+H_1(a)t+\dots)^p(b+H_1(b)t+\dots+H_p(b)t^p+\dots)
$$
which is $a^pH_p(b)+H_1(a)^pb$, so that
$$
f(ax)=*a^pH_p(b)+*H_1(a)^pb=a{*}H_p(b)+H_1(a){*}b=af(x)+H_1(a)x
$$
showing that $\delta(f)=H_1$. 
\end{proof}

 \begin{theorem}\label{T:main}
If a three-dimensional complete local domain of   \ch\ $p$ admits an F-invariant, integrable derivation, then it has a small MCM.
\end{theorem} 
 \begin{proof}
 By assumption, there exists a Hasse-Schmidt derivation $\mathbf H$  with $H_1=H_1^p$.  Let $\mathcal H$ be the collection of all F-components of $R$ (i.e., all summands of   $\frob*^nR$, for all $n$).   If we want to apply \Prop{P:FaddFind} to $\mathcal H$, so that we get a small MCM in view of \Prop{P:candep2}, then    we must show that any $Q\in\mathcal H$ is F-decomposable. By assumption, $Q$ is a summand of some $\frob*^nR$. By \Prop{P:HSFrob}, we have $H_1\in\KSker{\frob*^nR}$ and hence $H_1\in \KSker Q$ by \eqref{eq:KSsum}, so that $Q$ is decomposable by \Prop{P:FinvDind}.
\end{proof}

 \subsection*{$\op S_2$-ification and pseudo-perfect modules}
 Since $\can R{}$ is indecomposable (as $R$ is a domain), its endomorphism ring 
 $S:= \ndomod R{\can R{}}$ is a local ring (by \Thm{T:indloc}) satisfying Serre's condition $(\op S_2)$, called the \emph{$\op S_2$-ification} of $R$ 
 (see, for instance \cite[Theorem 3.2]{AoyCan} or  \cite{HHIndCan}).  Since $S$ is a finite $R$-module (contained in the field of fractions of $R$), any small MCM over $S$, is then also   a MCM over $R$. In short, if we want to do so, we may moreover assume that $R$ is an  \emph{$\op S_2$-domain}, i.e., satisfies Serre's condition $(\op S_2)$. 

Identifying $\can{\can M{}}{}$ with $ \hom R{\hom R M{\can R{}}}{\can R{}}$ via  \eqref{eq:toppcan}, for a module $M$, we have a canonical map $M\to \can{\can M{}}{}$, given by sending $x\in M$ to the \homo\  $\hom R M{\can R{}}\to \can R{}\colon \varphi\mapsto \varphi(x)$. If this map is an isomorphism, then we will call $M$ \emph{pseudo-perfect}.  
In view of \eqref{eq:toppcan}, we therefore showed that a complete local domain satisfies property $(\op S_2)$ \iff\ it is itself pseudo-perfect. It is not hard to show using \eqref{i:frobcommmat}, that any F-component of a pseudo-perfect module is again pseudo-perfect. Therefore, analyzing the above proofs,  we actually showed

\begin{corollary}\label{C:main}
If a three-dimensional complete local $\op S_2$-domain $R$   admits an F-invariant, integrable derivation, then some F-component of $R$ is a small MCM.\qed
\end{corollary}

\section{Application: pseudo-graded rings}
 Throughout this section, let $S:=\pow k{x_1,\dots,x_r}$   and let $A:=\pol S{\inv {x_1},\dots,\inv {x_r}}$ be the \emph{ring of Laurent series}. 
Any element $f\in A$ can be written as $f=\sum_{\tuple a} u_{\tuple a} x^{\tuple a}$,   with $u_{\tuple a}\in k$ and   ${\tuple a}\in\zet^r$ (with the usual convention that $x^{\tuple a}:=x_1^{a_1}\cdots x_r^{a_r}$), so that in addition, its \emph{support} $\op{supp}(f)$, that is to say, the set of all ${\tuple a}\in\zet^r$ such that $u_{\tuple a}\neq0$,  is contained in some translate $ \tuple c+\nat^r$, for some $\tuple c\in\zet^r$.

Fix a non-zero linear form $\lambda\in\hom{}{\zet^r}\zet$ and let $Z(\lambda)\subset \zet^r$ be its kernel. This form is represented by an $r$-tuple $\tuple l:=\rij lr$, so that $\lambda\rij ar=l_1a_1+\dots+l_ra_r$ for any $\rij ar\in\zet^r$. Put differently, if $\tuple e_i$ is the $i$-th standard basis element of $\zet^r$, then $l_i=\lambda(\tuple e_i)$. 
Let us say that an element $f\in A$ is \emph{$\lambda$-homogeneous}, if $\lambda$ is constant on its support $\op{supp}(f)$, or, equivalently, if $\op{supp}(f)\sub {\tuple a}+ Z(\lambda)$, for some ${\tuple a}\in \zet^r$, called the \emph{weight vector} of $f$. An ideal $I\sub S$ will be called a \emph{$\lambda$-ideal}, if it is generated by $\lambda$-homogeneous elements. Any complete local ring $R$ that can be realized as a quotient  $R:=S/I$ with $I$ a $\lambda$-ideal for some non-zero linear form $\lambda$, will be called  a \emph{pseudo-graded ring}. An example is the  completion of a standard graded ring at its irrelevant maximal ideal (with all $l_i=1$). To $\lambda$, we also associate the  derivation 
$$
\Delta_\lambda:=l_1x_1\partial_1+\dots+l_rx_r\partial_r
$$
where $\partial_i:=\partial/\partial x_i$ is the  $i$-th partial derivative on $S$, whence on $A$.  A quick calculation yields 
\begin{equation}\label{eq:dLmon}
\Delta_\lambda(x^{\tuple a})=\lambda({\tuple a})x^{\tuple a}
\end{equation}
for all ${\tuple a}\in \zet^r$. In particular, $\Delta_\lambda(f)=\lambda({\tuple a})f$, for any   $\lambda$-homogeneous element $f$ with weight vector ${\tuple a}$, showing that any $\lambda$-ideal is invariant under $\Delta_\lambda$. In particular, $\Delta_\lambda$ induces a derivation on the pseudo-graded quotient $R$. A $p$-fold iteration of \eqref{eq:dLmon} then yields 
$$
\Delta_\lambda^p(x^{\tuple a})=\lambda({\tuple a})^px^{\tuple a}=\lambda({\tuple a})x^{\tuple a}=\Delta_\lambda(x^{\tuple a})
$$
and since this holds for any monomial, we get $\Delta_\lambda^p=\Delta_\lambda$, that is to say, $\Delta_\lambda$ is F-invariant.

Our next goal is to show that $\Delta_\lambda$ is integrable, and to this end we will use generalized binomials. For $m,d\in\nat$, we have the so-called   \emph{negation rule} 
$$
\binomial{-m}d:=(-1)^d\binomial{m+d-1}d,
$$
and the usual additive rule (Pascal identity)
\begin{equation}\label{eq:addgenbin}
\binomial{z}d= \binomial{z-1}d+\binomial {z-1}{d-1}\qquad\text{for all $z\in\zet$ and $d>0$.}
\end{equation}

\begin{proposition}\label{P:binomintdeltaid}
There exists a Hasse-Schmidt derivation $\mathbf H_\lambda=(1,\Delta_\lambda,H_2,\dots)$ on $S$, which leaves every $\lambda$-ideal invariant. In particular, any pseudo-graded local ring admits an F-invariant, integrable derivation. 
\end{proposition} 
\begin{proof}
We will be more precise and show that  the $k$-linear maps given by 
\begin{equation}\label{eq:Hbinomn}
H_n(\sum u_{\tuple a} x^{\tuple a}):=\sum_{\tuple a} u_{\tuple a}\binomial{\lambda({\tuple a})}nx^{\tuple a}
\end{equation} 
for  $\sum_{\tuple a} u_{\tuple a} x^{\tuple a}\in A$, yield  a Hasse-Schmidt derivation 
 on the ring of Laurent series $A$. 
Note that   $H_1=\Delta_\lambda$ does satisfy \eqref{eq:Hbinomn} in view of   \eqref{eq:dLmon}. 
To verify \eqref{eq:mixLeib}, let $f=\sum u_{\tuple a} x^{\tuple a}$ and $g=\sum_{\tuple b} v_{\tuple b} x^{\tuple b}$ be in $A$. We get
$$
\begin{aligned}
H_n(fg)&=H_n(\sum_{{\tuple a},{\tuple b}}u_{\tuple a} v_{\tuple b} x^{{\tuple a}+{\tuple b}})=\sum_{{\tuple a},{\tuple b}}u_{\tuple a} v_{\tuple b} H_n(x^{{\tuple a}+{\tuple b}})\\
&= \sum_{{\tuple a},{\tuple b}}u_{\tuple a} v_{\tuple b} \binomial{\lambda({\tuple a}+{\tuple b})}nx^{{\tuple a}+{\tuple b}}\overset{(\textsl{CH})}=\sum_{i,{\tuple a},{\tuple b}}  u_{\tuple a} v_{\tuple b} \binomial{\lambda({\tuple a})}i\binomial{\lambda({\tuple b})}{n-i}  x^{\tuple a} x^{\tuple b}\\
&=\sum_{i,{\tuple a},{\tuple b}}  u_{\tuple a} v_{\tuple b} H_i(x^{\tuple a})H_{n-i}(x^{\tuple b}) =\sum_i H_i(\sum u_{\tuple a} x^{\tuple a})H_{n-i}(\sum_{\tuple b} v_{\tuple b} x^{\tuple b})\\
&=\sum_{i=0}^n H_i(f)H_{n-i}(g)
\end{aligned}
$$
%
%
%
where   we used the Chu-Vandermonde identity for binomial coefficients
$$
\binomial{v+w}n=\sum_{i=0}^n \binomial{v}i\binomial{w}{n-i}\qquad\qquad \thetag {\textsl{CH}}
$$
for $v,w\in\zet$, since  
$\lambda({\tuple a}+{\tuple b})=\lambda({\tuple a})+\lambda({\tuple b})$.

It  follows from \eqref{eq:Hbinomn} that $\mathbf H_\lambda$ leaves $S$ invariant, so that it is a Hasse-Schmidt derivation on $S$. 
To show that $\mathbf H_\lambda$ leaves any $\lambda$-ideal of $S$ invariant, let $A_\lambda\sub A$ be the subring of all Laurent series whose support lies in the kernel $Z(f)$ of $\lambda$. It follows from \eqref{eq:Hbinomn} that each $H_n$ is identically zero on $A_\lambda$. Therefore, each $H_n$ is $A_\lambda$-linear, since for $f\in A$ and $g\in A_\lambda$, we have $H_n(fg)=fH_n(g)+H_1(f)H_{n-1}(g)+\dots+H_n(f)g$ in which all but the last term are zero, so that $H_n(fg)=gH_n(f)$.  
%
%
%
%
Let $f\in S$ be $\lambda$-homogeneous. By definition, it is of the form  $f=x^{\tuple a} g$, with $g\in A_\lambda$. Hence, using \eqref{eq:dLmon}, we get 
$$
H_n(f)=H_n(x^{\tuple a} g)= gH_n(x^{\tuple a})=g\binomial{\lambda({\tuple a})}nx^{\tuple a}  =\binomial{\lambda({\tuple a})}nf.
$$
Hence  any $\lambda$-ideal is invariant under $H_n$, for all $n$.  
To prove the last assertion, let $R$ be pseudo-graded, say,  $R=S/I$ for some linear form $\lambda\colon\zet^r\to \zet$ and some $\lambda$-ideal $I\sub S$. Since $I$ is  invariant under $\mathbf H_\lambda$, the latter induces a Hasse-Schmidt derivation on $R$.
\end{proof} 

By \Cor{C:HSass}, any component of a pseudo-graded ring, i.e., any quotient  by an associated prime, also admits an F-invariant integrable derivation and  \Thm{T:main} yields:

\begin{theorem}\label{T:psgr3}
Any three-dimensional component of a pseudo-graded local ring in positive \ch\ admits a small MCM.\qed
\end{theorem}

\subsection*{Examples of pseudo-graded rings}
To an element $f$  in $S$ or in $\pol kx$, we associate its lattice  $\Lambda(f)\sub\zet^r$, as the subgroup generated by all differences of elements in $\op{supp}(f)$ (with the convention that $\Lambda(f)=0$ when $f$ is a monomial).  We have inclusions
\begin{equation}\label{eq:Lfg}
\Lambda(fg), \Lambda(f+g)\sub \Lambda(f)+\Lambda(g),
\end{equation} 
 for all $f,g$. 
The \emph{rank} $\op{rk}(f)$ of $f$ is defined to be the rank of  $\Lambda(f)$, that is to say, the vector space dimension of $\Lambda(f)\tensor \mathbb Q$; in particular, $\op{rk}(f)\leq r$. A \emph{binomial} is an element whose support consists of two elements, and so binomials have rank one.  We may extend the above to  ideals $I$ in  $S$ or $\pol kx$: define $\Lambda(I):=\Lambda(f_1)+\dots+\Lambda(f_s)$ for $I=\rij fs$. By \eqref{eq:Lfg}, this does not depend on the choice of generators, and we have in particular that  $\Lambda(I+J)=\Lambda(I)+\Lambda(J)$ for any two ideals $I,J$. We then define $\op{rk}(I)$ as the rank of $\Lambda(I)$.   More generally, if $R$ is the quotient of $S$ or $\pol kx$ by some ideal, then we define the rank of an ideal $I\sub R$ as the smallest rank of a lifting of $I$ to $S$ or $\pol kx$. 

\begin{lemma}\label{L:rkpg}
If $R=S/I$ with $\op{rk}(I)<r$, then $R$ is pseudo-graded. 
\end{lemma}
\begin{proof}
By assumption, there is some $\rij lr\in\zet ^r$ which is orthogonal to $\Lambda(I)$, and it is now easy to see that $I$ is a $\lambda$-ideal for the linear form $\lambda:=l_1z_1+\dots+l_rz_r$. 
\end{proof}

By the theory of toric varieties, or using  \cite[Theorem 2.1]{EisStu} for the more general case, we  have (recall that $r$ is the length of the tuple of variables $x$):

\begin{theorem}\label{T:EisStu}
If $I\sub\pol kx$ is generated by binomials, then $\pol kx/I$ has dimension $r-\op{rk}(I)$.  Moreover, if $I$ is also prime, the quotient $\pol kx/I$ is the coordinate ring of  a toric variety in $\mathbb A_k^r$, and any coordinate ring of a toric variety is obtained this way.\qed
\end{theorem}

By a \emph{toric singularity} (of embedding dimension $r$), we mean the local ring of a (singular) point on a toric variety (in $\mathbb A_k^r$), and the completion of a such a ring is then called an \emph{analytic toric singularity}.  Recall that the normalization of a toric variety is again a toric variety which is in addition \CM\ (\cite{HoToric}). Since the normalization is a finite extension, any analytic toric singularity therefore admits a small MCM. Immediately from \Thm{T:EisStu} and \Lem{L:rkpg} we see that an   analytic toric singularity is pseudo-graded.

\begin{theorem}\label{T:singtoric}
Let $T$ be  analytic toric singularity, or more generally, a pseudo-graded complete local domain, and let $V\sub\op{Spec}T$ be its singular (respectively, non-\CM, non-Gorenstein, non-normal) locus. If $\pr$ defines a three-dimensional  irreducible component of $V$, then $T/\pr$ admits a small MCM. 
\end{theorem} 
\begin{proof}
By \Prop{P:binomintdeltaid}, there exists a Hasse-Schmidt derivation $\mathbf H=(H_i)_i$ on $T$ with $H_1^p=H_1$. By \Cor{C:HSass}\eqref{i:HSsing}, this $\mathbf H$   descends to a Hasse-Schmidt derivation on $T/\pr$ and so we are done by  \Thm{T:main}. 
\end{proof}

\begin{proposition}\label{P:fewnom}
Let $T$ be  a $d$-dimensional analytic toric singularity. If  $I\sub T$ is an ideal of rank strictly less than $d$, then $T/I$ is pseudo-graded. 
\end{proposition} 
\begin{proof}
Writing again $I$ for a lift to $S:=\pow kx$ of minimal rank, we must show that $I+J$ is a $\lambda$-ideal for some linear form $\lambda$, where $J$ is the  ideal generated by binomials such that  $T=\pow kx/J$. By \Thm{T:EisStu}, the  rank of $J$ is $r-d$. Since $\Lambda(I+J)=\Lambda(I)+\Lambda(J)$ and $\Lambda(I)$ has rank at most $d-1$ by assumption,  the rank of $\Lambda(I+J)$ is at most $r-1$, so that we are done by \Lem{L:rkpg}.
\end{proof} 

\begin{corollary}\label{C:torcyl}
Let $T$ be a $d$-dimensional analytic toric singularity,  and let $C\sub\op{Spec}(S)$ be a cylinder with base inside a $(d-1)$-dimensional coordinate hyperplane. Then the coordinate ring of any three-dimensional irreducible component of $C\cap \op{Spec}(T)$ admits a small MCM.
\end{corollary}
\begin{proof}
By assumption, there exists a subset of the $x$-variables of size $d-1$ defining the ideal $I\sub T$ of $C$. In particular, $\op{rk}(I)\leq d-1$ and so $T/I$, the coordinate ring of $C\cap \op{Spec}(T)$  is pseudo-graded by \Prop{P:fewnom}, and the result now follows from \Thm{T:psgr3}. 
\end{proof}

By an \emph{$m$-nomial} in a quotient of $S$ or $\pol kx$, we mean the image $f$ of an element whose support has cardinality $m$. Note that then $\op{rk}(f)<m$, and any element satisfying the latter inequality is called a \emph{pseudo-$m$-nomial}.
For instance, the trinomial
$$
f:=u_0x^2z^4+u_1xy^2z^2+u_2y^4
$$
is in fact a pseudo-binomial as $\Lambda(f)$ is generated by $(1,-2,2)$. Let us call a \homo\ of complete local rings with residue field $k$  \emph{(pseudo-)$m$-nomial}, if it is given by (pseudo-)$m$-nomials, that is to say, induced by a \homo\ $\pow ky\to\pow kz\colon y_i\mapsto f_i$, where each $f_i$ is a (pseudo-)$m$-nomial. 

\begin{corollary}\label{C:invhomotor}
Let $T$ be a $d$-dimensional analytic toric singularity, $I\sub S$ a monomial  ideal, and $S\to T$ a pseudo-$m$-nomial \homo. If $r(m-1)<d$, then $T/IT$ is pseudo-graded. In particular, if $\pr$ is  a three-dimensional associated prime of $IT$, then $T/\pr$ admits a small MCM.
\end{corollary}  
\begin{proof}
Write $T$ as a quotient of some $\pow kz$, with $z=\rij zs$, and let $S\to T$ be given by sending $x_i$ to the pseudo-$m$-nomial $f_i\in\pow kz$, for $i=\range 1r$. By assumption, $\Lambda_i:=\Lambda(f_i)\sub\zet^s$ has rank at most $m-1$. Let $B:=\pol{\pow kz}{\inv {z_1},\dots,\inv {z_s}}$ and  write $f_i=z^{\tuple a_i}g_i$, for some $\tuple a_i\in\zet^s$ and  some $g_i\in B$ with support in $\Lambda_i$. Let $\sigma\colon \zet^r\to\zet^s$ be the linear map given by $\rij br\mapsto b_1\tuple a_1+\dots+b_r \tuple a_r$. 
It follows that the image of a monomial $f^{\tuple b}$ in $T$, with $\tuple b\in\zet^r$,   is equal to  $z^{\sigma(\tuple b)}g_{\tuple b}$, for  some $g_{\tuple b}\in B$  with $\op{Supp}(g_{\tuple b})\sub \Lambda:=\Lambda_1+\dots+\Lambda_r$. 
%
%
%
This shows that $\Lambda(IT)\sub\Lambda$ and since the latter has rank at most $r(m-1)$,  we are done by \Prop{P:fewnom} and    \Thm{T:psgr3}.
%
%
%
\end{proof}
\begin{remark}\label{R:invhomotor}
More generally, the result still holds if $I\sub S$ is an arbitrary ideal such that  $\op{rk}(I)<d-r(m-1)$. Indeed, $I$ is then generated by elements of the form $h=x^{\tuple b}\sum u_{\tuple c}x^{\tuple c}$, with $\tuple c\in \Theta:=\Lambda(I)$, $u_{\tuple c}\in k$ and $\tuple b\in \zet^r$. The image of $h$ in $T$ is equal to 
$
z^{\sigma(\tuple b)}g_{\tuple b}\sum z^{\sigma(\tuple c)}g_{\tuple c}
$
and so its  support is inside $\sigma(\tuple b)+\sigma(\Theta)+\Lambda$, showing that $\Lambda(IT)\sub\sigma(\Theta)+\Lambda$. As $\op{rk}(\sigma(\Theta))\leq\op{rk}(\Theta)=\op{rk}(I)$, we get $\op{rk}(IT)<d$.
\end{remark}  

\begin{corollary}\label{C:tor4quad}
Let $T$ be a $d$-dimensional  analytic toric singularity. If $\pr\sub T$ is  a three-dimensional prime, then $T/\pr$ admits a small MCM in the two following cases:
\begin{enumerate}
\item $d=4$ and $\pr$ contains a non-zero (pseudo-)quadrinomial;
\item $d=5$ and $\pr$ contains a height two ideal generated by two (pseudo-)trinomials. 
\end{enumerate}
\end{corollary}
\begin{proof}
Apply \Cor{C:invhomotor}, where $I$ is the ideal generated by the variables, and   $S\to T$ is given, in the case $d=4$, by the quadrinomial with $r=1$, and in case $d=5$, by the two trinomials with $r=2$. 
%
%
\end{proof} 
\begin{remark}\label{R:tor4quad}
More generally, if there exists an ideal $I\sub T$ of height $d-3$   generated either by $d-4$ pseudo-binomials and one pseudo-quadrinomial or by $d-5$ pseudo-binomials and two pseudo-trinomials, then any $T/\pr$ admits a small MCM for any minimal prime $\pr$ of $I$. The same is true if instead in the above, the number of pseudo-binomials is arbitrary and we also allow monomial generators instead (as long as the height of the ideal is $d-3$).
\end{remark}

\section{Appendix: hereditary strongly local algebras}
Recall that a non-commutative ring $A$ is called \emph{local} if it has a unique maximal left ideal, which is then also the unique maximal right ideal, and this is then also the Jacobson radical $\op{rad}( A)$ (see, for instance, \cite{LamNon}). It follows that $A$ is local \iff\ $A/\op{rad}(A)$ is a division ring. Therefore,  $A$ is local \iff, for any $f\in A$, either $f$ or $1-f$ is a unit. We call  $A$ \emph{strongly local} if $A/\op{rad}(A)$ is a field, called its \emph{residue field}. Finally, if $R$ is a commutative ring and $A$ an $R$-algebra, then we call   $A$  \emph{hereditary (strongly) local} over   $R$, if any $R$-subalgebra of $A$ is (strongly) local. The following is folklore:

\begin{theorem}\label{T:indloc}
Let $\mathcal E:=\ndomod RQ$ be the endomorphism ring of a   finitely generated $R$-module $Q$ over a Henselian local ring $R$ with algebraically closed residue field $k$. The following are equivalent 
\begin{enumerate}
\addtocounter{enumi}4
\item \label{i:indec} $Q$ is indecomposable;
\item\label{i:idem}  $\mathcal E$ has no non-trivial idempotents;
\item\label{i:loc} $\mathcal E$ is local;
\item\label{i:sloc} $\mathcal E$ is strongly local with residue field $k$;
\item\label{i:hloc} $\mathcal E$ is hereditary strongly local over $R$.
\end{enumerate}
\end{theorem}
\begin{proof}
We only need to show that the first condition implies the last. Let $A\sub\mathcal E$ be an $R$-subalgebra and take some $f\in A$. Let $S\sub A$ be the $R$-subalgebra generated by $f$. Since $S$ is commutative and $R$ is  Henselian, $S$ must be a direct sum $S_1\oplus\dots\oplus S_m$ of local rings, and since $k$ is algebraically closed, they all have residue field $k$. But then $m$ must be equal to $1$ lest we violate \eqref{i:idem}. There is a unique $u\in k$ such that $f-u$ lies in the maximal ideal of $S$, whence in the radical of  $A$, and so $A$ is strongly local with residue field $k$.
\end{proof} 

In the sequel, let $R$ be  a complete local ring of   \ch\ $p$ with algebraic residue field $k$ and let $M$ be a finitely generated $R$-module. 
Although we  did not   use   it explicitly, the following is the underlying reason for introducing differential algebra into the problem: let $\mathcal D(M)\sub\ndomod kM$ denote the subring of  differential operators (=the endomorphisms of finite order)  on $M$ (see \S\ref{s:diff}); Grothendieck showed (see \cite{Yek}) that it consists precisely of the $\powers Rn$-linear endomorphisms, for some $n$, where $\powers Rn\sub R$ is the subring of $p^n$-th powers of elements of $R$. 

\begin{corollary}\label{C:Findloc}
The ring of differential operators $\mathcal D(M)$ is hereditary strongly local \iff\ $M$ is F-indecomposable. 
\end{corollary} 
\begin{proof}
Note that $\ndomod {\powers Rn}M\iso \ndomod R{\frob*^nM}$, and so if $\mathcal D(M)$ is local, whence has no non-trivial idempotents,  neither therefore does  the subring $ \ndomod R{\frob*^nM}$, proving that $\frob*^nM$ is indecomposable. Conversely, suppose $M$ is F-indecomposable, and let $S\sub\mathcal D(M)$ be an arbitrary $R$-subalgebra. Since $\frob*^nM$ is by assumption indecomposable, the subalgebra $S_n:=S\cap\ndomod {\powers Rn}M$ is strongly local by \Thm{T:indloc}, for all $n$. It is not hard to see that   $S=\bigcup_n S_n$ is then also strongly local.   
\end{proof}  

Let us say that a   submodule $H\sub\mathcal D$ is   \emph{F-closed} if $h^p\in H$, for all $h\in H$. 

\begin{proposition}\label{P:idemgrdiff}
If there exists a finitely generated F-closed $R$-submodule $H\sub\mathcal D$ containing $R$,  some $f\notin H$, and an $l$ such that $f^{p^l}-f\in H$,     then $M$ is F-decomposable.
\end{proposition}
\begin{proof}
Put $q:=p^l$. Towards a contradiction, suppose $M$ is F-indecomposable.   Since $\mathcal D$ is then hereditary strongly local by \Cor{C:Findloc},  we can find a (unique) $u\in k$ such that $g:=f-u$ is not a unit, and since then also $g^q- g\in H$,  we may assume form the start that $f$ is not  a unit. By \Cor{C:Findloc}, we may choose $d$ large enough so that   $f\in A:=\ndomod {\powers Rd}M\sub \mathcal D$ and $H\sub A$. Since $A$ is then a local $R$-algebra, $f$ lies in its Jacobson radical $\mathfrak n:=\op{rad} A$. Since $A$ is finite over $R$, so too is $A/\maxim A$   over $k$, and hence some  power of $\mathfrak n$ lies in $\maxim A$, say, $\mathfrak n^l\sub\maxim A$. It follows that $f^{ln}\in\maxim^nA$, for all $n$. Since $H$ is F-closed and $ f^q-f$ lies  in $H$,   so does   $(f^q-f)^q=f^{2q}-f^q$, whence also $f^{2q}- f$. Continuing this way, we get $f^{lnq}- f\in H$, whence  $f\in H+\maxim^{nq}A$, for all $n$. Since $H$, being finitely generated, is $\maxim$-adically closed in $A$, we actually get $f\in H$, contradiction.
\end{proof}

\Prop{P:FinvDind} is now an immediate consequence of \Prop{P:idemgrdiff}: take  $f$   as in the former's proof and  let $H:=\ndomod RM$, so that $f^p-f\in H$.


\begin{thebibliography}{10}

\bibitem{AndDS}
Yves Andr\'{e}, \emph{La conjecture du facteur direct}, Publ. Math. Inst.
  Hautes \'{E}tudes Sci. \textbf{127} (2018), 71--93.

\bibitem{AndMix}
\bysame, \emph{Singularities in mixed characteristic. {T}he perfectoid
  approach}, Jpn. J. Math. \textbf{14} (2019), no.~2, 231--247.

\bibitem{AoyCan}
Y{\^o}ichi Aoyama, \emph{Some basic results on canonical modules}, J. Math.
  Kyoto Univ. \textbf{23} (1983), no.~1, 85--94. \MR{692731 (84i:13015)}

\bibitem{BH}
W.~Bruns and J.~Herzog, \emph{Cohen-{M}acaulay rings}, Cambridge University
  Press, Cambridge, 1993.

\bibitem{EisStu}
David Eisenbud and Bernd Sturmfels, \emph{Binomial ideals}, Duke Math. J.
  \textbf{84} (1996), no.~1, 1--45.

\bibitem{HoToric}
M.~Hochster, \emph{Rings of invariants of tori, {C}ohen-{M}acaulay rings
  generated by monomials, and polytopes}, Ann. of Math. (2) \textbf{96} (1972),
  318--337.

\bibitem{HoCurr}
M.~Hochster, \emph{Current state of the homological conjectures}, Tech. report,
  University of Utah,
  http://www.math.utah.edu/vigre/minicourses/algebra/hochster.pdf, 2004.

\bibitem{HHbigCM}
M.~Hochster and C.~Huneke, \emph{Infinite integral extensions and big
  {C}ohen-{M}acaulay algebras}, Ann. of Math. \textbf{135} (1992), 53--89.

\bibitem{HHIndCan}
\bysame, \emph{Indecomposable canonical modules and connectedness}, Commutative
  algebra: syzygies, multiplicities, and birational algebra ({S}outh {H}adley,
  {MA}, 1992), Contemp. Math., vol. 159, Amer. Math. Soc., Providence, RI,
  1994, pp.~197--208.

\bibitem{LamNon}
T.~Y. Lam, \emph{A first course in noncommutative rings}, Graduate Texts in
  Mathematics, vol. 131, Springer-Verlag, New York, 1991.

\bibitem{Mats}
Hideyuki Matsumura, \emph{Commutative ring theory}, Cambridge University Press,
  Cambridge, 1986.

\bibitem{SchBCM}
Hans Schoutens, \emph{Canonical big {C}ohen-{M}acaulay algebras and rational
  singularities}, Illinois J. Math. \textbf{48} (2004), 131--150.

\bibitem{SchClassSing}
\bysame, \emph{Classifying singularities up to analytic extensions of scalars
  is smooth}, Annals of Pure and Applied Logic \textbf{162} (2011), 836--852.

\bibitem{SchMCMFsplit}
\bysame, \emph{Hochster's small {MCM} conjecture for three-dimensional weakly
  {F}-split rings}, Communications in Algebra \textbf{45} (2017), 262--274.

\bibitem{WeiHom}
Charles~A. Weibel, \emph{An introduction to homological algebra}, Cambridge
  Studies in Advanced Mathematics, vol.~38, Cambridge University Press,
  Cambridge, 1994.

\bibitem{Yek}
Amnon Yekutieli, \emph{An explicit construction of the {G}rothendieck residue
  complex}, Ast\'erisque (1992), no.~208, 127, With an appendix by Pramathanath
  Sastry.

\end{thebibliography}
\providecommand{\bysame}{\leavevmode\hbox to3em{\hrulefill}\thinspace}
\providecommand{\MR}{\relax\ifhmode\unskip\space\fi MR }
\providecommand{\MRhref}[2]{%
  \href{http://www.ams.org/mathscinet-getitem?mr=#1}{#2}
}
\providecommand{\href}[2]{#2}

\end{document}